\documentclass[12pt]{article}
\usepackage{amsmath,amsfonts,amsthm,amssymb,blindtext
,graphicx,pifont,fncylab,theoremref,sectsty,tocloft,titlesec,enumitem,
fmtcount,titletoc}
\date{}
\newtheorem{thm}{Theorem}
\newtheorem{lem}{Lemma}

\newtheorem{prp}{Proposition}
\newcommand{\C}{\Bbb C}

\begin{document}

\baselineskip 20pt
\title{Arakeljan's Theorem with bounded approximants}

\maketitle
\begin{center}

\author{Spyros Pasias}\\
\indent Orcid:0000-0001-6611-0155\\
\indent spyrospasias2393@outlook.com\\
\indent Paphos
\end{center}
\begin{abstract}
Arakeljan's Theorem provides conditions on a relatively closed subset $F$ of a domain $G\subset\mathbb{C}$, such that any continuous function $f:F\rightarrow\mathbb{C}$ that is analytic in $F^\circ$, can be approximated by analytic functions defined on $G$. In this paper we will extend Arakeljan's theorem by adding the extra requirement that the analytic functions that approximate $f$ may also be chosen to be bounded on a closed set $K\subset G.$ In \cite{RU} the same problem has been considered but for the specific case that $G=\mathbb{C}$.

\end{abstract}
\begin{center}

\end{center}
\pagebreak
\section*{Arakeljan sets and Arakeljan's Theorem}

\indent In this paper we will generalize an important theorem in the field of Complex approximation, namely we will generalize \emph{Arakeljan's Theorem}. A more well known theorem is \emph{Mergelyan's Theorem} which states that any function that is continuous on a compact set $F\subset\mathbb{C}$ and analytic on $F^o$ can be uniformly approximated by polynomials on $F$, provided our set $F$ has a connected compliment.\\
\indent {Arakaljan's Theorem} on the other hand deals with analytic functions defined in more general sets $F$ that are not necessarily simply connected nor compact. Consequently we don't expect the approximation to be by polynomials. \\
\indent To this end let $G\subset\mathbb{C}$ be an arbitrary domain and $F\subset G$ a relatively closed subset of $G$. {Arakeljan's Theorem} provides conditions that determine whether a function $f\in A(F)$ can be approximated on $F$ by functions $g\in Hol(G)$ 
\begin{align*}
Hol(G)=\lbrace  g: g\text{ analytic in }G\rbrace \\
 A(F)=\lbrace  f: f\text{ analytic on }F^o,\text{ continuous on F}\rbrace
\end{align*}

\begin{thm}[Arakeljan's Theorem]
Let $F$ and $G$ be as above, then any $f\in A(F)$ can be uniformly approximated by functions $g\in Hol(G)$ if and only if $F$ is an Arakeljan set.
\end{thm}

\noindent In \cite{GA} an Arakeljan set is defined differently than the definition we will use in this paper. In this paper we will use another definition which is also used in \cite{RU} to prove an extension of Arakeljan's Theorem for the case $G=\mathbb{C}$.\\

\noindent\textbf{Definition 1.}
Let $G\subset \mathbb{C}$ be an arbitrary domain and $F\subset G$ be a relatively closed subset of $G$. A connected component $g$ of $G\setminus F$ is called a $G-$hole of $F$ if: \begin{enumerate}
\item$\partial g\cap \partial G= \emptyset$. 
\item$\ g \text{ is bounded (a hole)}$.
 \end{enumerate}

\noindent\textbf{Remark 1.} We note that the definition of a $G-$hole can be restated as follows: A connected component $g$ of $G\setminus F$ is called a $G-$hole of $F$ if $g$ can be enclosed in a compact set $L\subset G$. Even though this alternative definition is more intuitive than the one introduced above; for the purpose of the proofs that follow the definition above is more useful.\vspace{0.12in}

\noindent Now that we defined what a hole is in the setting of Arakeljan's theorem on an arbitrary domain, we are ready to define Arakeljan sets.\vspace{0.12in}

\noindent\textbf{Definition 2.} Let $G$ be an arbitrary domain and let $F$ be a relatively closed subset, then we call $F$ an Arakeljan set if: \begin{enumerate}
\item The set $F$ has no $G-$holes.
\item For any connected compact set $K\subset G$ such that $\partial K$ is union of jordan curves the set $H\equiv\lbrace \bigcup \lbrace h\rbrace :h \text{ is a $G$-hole of }F\cup K\rbrace$ is bounded and $\partial H\cap \partial G=\emptyset$.
\end{enumerate}

\noindent\textbf{Remark 2.} By condition $1$ of {Definition 2} it follows that if $H$ has only a finite number of $G-$holes and $F$ has no $G-$holes, then $F$ is an Arakeljan set.\vspace{0.12in}

\section*{Description of Arakelan sets}
 In the proof of {Arakeljan's Theorem} in (\cite{GA}, p. 142-144) the definition used for $F$ to be an Arakeljan set is different than {Definition 2}. The definition used in (\cite{GA}, p. 142) is more topological but it turns out that both definitions are equivalent.\\
\indent  Now, in order to compare {Definition 2} with the more topologically flavored definition presented in (\cite{GA}, p. 142), we will need to some tools from topology. To this end let us recall the \emph{Alexandroff compactification} of an arbitrary domain $G\subset \mathbb{C}$. The set $G^*$ is defined by introducing the point $\infty$ so that $G^*=G\cup \lbrace\infty\rbrace$. The topology of $G^*$ is defined to consist from the open sets of $G$ and in addition all the sets that are compliments of compact subsets $K\subset G$. Under this topology one may check that $G^*$ is compact.\vspace{0.12in}
 
 Now we are ready to provide the definition of an Araklejan set as presented in (\cite{GA}, p. 142).\vspace{0.12in}

\noindent\textbf{Definition 3.} The conditions given on the set $F$ to be an Arakeljan set are: \begin{enumerate}
\item$G^*\setminus F$ is connected.
\item$G^*\setminus F$ is locally connected at $\infty$.
\end{enumerate}

 The following proposition introduced in (\cite{GA}, p. 133) provides us with an intuitive approach for determining whether $G^*\setminus F$ is connected.

\begin{prp}
The space $G^*\setminus F$ is connected if and only if each component of the open set $G\setminus F$ has an accumulation point on $\partial G$ or is unbounded.
\end{prp}

\begin{proof}
The proof can be found in (\cite{GA}, p. 133-134).
\end{proof}

Even though the conditions of {Definition 3} given above seem different from the ones in {Definition 2}, actually they are the same.
\begin{prp}
Definition 2 and Definition 3 describing Arakeljan sets are equivalent.
\end{prp}  

\begin{proof}
Indeed, let $G$ be an arbitrary domain and $F$ a relatively closed subset of $G$.
By {Proposition 1}, the set $G^*\setminus F$ is connected if and only if each component of the open set $G\setminus F$ has an accumulation point on $\partial G$ or is unbounded. Therefore, $G\setminus F$ is connected if and if every hole (component of $G\setminus F$) is not a $G-$hole. Therefore, conditions 1 of both definitions are equivalent.
Now we will show that conditions 2 of both definitions are equivalent as well.
Indeed, suppose that for any connected compact set $K\subset G$ such that $\partial K$ is the union of jordan curves, the set $H\equiv\lbrace \bigcup \lbrace h\rbrace :h \text{ is a $G$-hole of }F\cup K\rbrace$ is bounded and $\partial H\cap \partial G=\emptyset$. We will show that this implies that $G^*\setminus F$ is locally connected at $\infty$.
Now, recall that $G^*\setminus F$ being locally connected at $\infty$ by definition requires that every open set of $G^*\setminus F$ containing $\infty$ contains a connected sub-neighborhood containing $\infty$. However, every open set in $G^*\setminus F$ is simply the 
complement of a compact set $K\subset G\setminus F$. Suppose for a contradiction that $G^*\setminus F$ is not locally connected. Then there exists a compact set $K\subset G\setminus F$ such that the set  $G\setminus (F\cup K)$ has infinite components that eventually accumulate on $\partial G$ or are unbounded. Then since $G$ is open and path connected we can easily construct a connected compact set $L\supset K$ such that $\partial L$ is the union of jordan curves, and such that the set $H\equiv\lbrace \bigcup \lbrace h\rbrace :h \text{ is a $G$-hole of }F\cup L\rbrace$ is either unbounded or $\partial H\cap\partial G\neq\emptyset$, i.e a contradiction. The other direction is obvious since for any compact set $K\subset G$ that fails condition 2 of definition 2, the neighborhood $G^*\setminus(F\cup K)\subset G^*\setminus F$ is not locally connected at $\infty$. The proof is complete.
\end{proof}
\pagebreak

 \section*{Extending Arakeljan's Theorem}
 In \cite{RU} the specific case where $G=\mathbb{C}$ has been extended to the following:

\begin{thm}[Extending Arakeljan's Theorem when $G=\mathbb{C}$]
Let $F$ and $C$ be closed sets in the complex plane, with $C\not=\mathbb{C}$. In order that every function $f\in A(F)$ can be approximated uniformly on $F$ by entire functions, each of which is bounded on $C$, it is necessary and sufficient that: \begin{enumerate}
\item $F$ be an Arakeljan set.
\item There exists an Arakeljan set $C_1$, $C_1\not=\mathbb{C}$, so that $C\subset C_1$ and $F\cap C_1$ is a bounded set.
\end{enumerate}
\end{thm}

 The proof of the extended version of Arakeljan's Theorem for $G=\mathbb{C}$ stated above is based on the following two lemmas: 

\begin{lem} Let $f$ be an entire function, and let $C$ a closed subset in $G$, $C\not=G$. In order that $f$ be bounded on $C$ it is necessary and sufficient that there exists a closed Arakeljan set $C_1\subset G$ so that $C\subset C_1$ and $f$ is bounded on $C_1$.
\end{lem}

\begin{lem}
The union of two disjoint Arakeljan sets $E$ and $F$ is again an Arakeljan set.
\end{lem}

Our goal now is to move past the restriction $G=\mathbb{C}$ and prove the following extended version of Arakeljan's Theorem.

\begin{thm}[Main Theorem]
Let $G\subset \mathbb{C}$ be an arbitrary domain and suppose $F$ is a relatively closed subset of $G$, and $C$ is a closed subset in $G$ such that \(F\) and \(C\) are \(G-\)hole independent. In order that every function $f\in A(F)$ can be approximated uniformly on $F$ by functions $g\in Hol(G)$, each of which is bounded on $C$, it is necessary and sufficient that: \begin{enumerate}
\item $F$ is an Arakeljan set.
\item There exists a closed Arakeljan set $C_1\not=G$, so that $C\subset C_1$ and $F\cap C_1$ is a bounded set.
\end{enumerate}
\end{thm}

\section*{Proofs}
 In order to prove the Extension of Arakeljan's Theorem we have to generalize {Lemma 1} and Lemma 2 for a general domain $G$. The reason we did not use the more intuitive definition for $G-$holes (see {Remark 1}) is so that we can adjust the proofs presented in \cite{RU} to fit the case of an arbitrary domain $G$.
 
\noindent We shall prove a generalization of Lemma 1 of Rubel-Danielyan.

\begin{lem}\label{level}
If $f$ is holomorphic on a Riemann surface $\Omega$ and $M\ge 0,$ then the level set $L_M=\{z:|f(z)|= M\}$ is locally connected. 
\end{lem}

\begin{proof}If $f$ is constant or if $M=0,$ the assertion is clear. We thus assume that $f$ is non-constant and $M>0.$ 

By a Moebius change of variables in the target domain, we may instead show that the level set $Y_0 = \{z: \Im f(z)=0\}$ is locally connected.  
 Fix a point in  in a local coordinate of $\Omega$ which me may assume to be zero. We may also assume that $f(0)=0.$ Let $n$ be the multiplicity of the zero at $0.$ 
There is a homeomorphism $h(z)=w$ mapping a neighbourhood of $z=0$ to the disc $|w|<1,$  such that $f(z)=[h(z)]^n=w^n.$ Thus, locally near $0$, $Y_0$ is topologically equivalent to $n$ radial segments emanating from $0$, which is connected.   
\end{proof}

\begin{lem}\label{level Arakelyan}.  If $f$ is holomorphic on a Riemann surface and $M\ge 0,$ then the sublevel set $S_M=\{z:|f(z)|\le M\}$ is an $\Omega$-Arakelyan set.  
\end{lem}

 For $\Omega=\C,$ this is due to Danielyan and Schmieder.  
\begin{proof}
If $f$ is constant,  the conjecture is trivial, so we suppose $f$ is not constant. 

By the maximum principle,  $S_M$ has no holes.

If $S_M$ is not locally connected at $\infty,$ we can show that $\partial S_M$  is not locally connected.  However, $\partial S_M$ is the level set $L_M$ which is locally connected by Lemma \ref{level}.  This contradiction proves the Lemma. 
\end{proof}

The following lemma is then a corollary. 

\begin{lem}
A holomorphic function $f$ on a Riemann surface $\Omega$ is bounded on a closed set $F\subset\Omega$, if and only if $F$ is contained in an $\Omega$-Arakelian set on which $f$ is bounded.   
\end{lem}

For $\Omega=\C,$ this is  Lemma 1 in Danielyan-Rubel. 
 \noindent Now, lets prove the analogue of {lemma 2} on an arbitrary domain $G$.

\begin{lem}
Let $G$ be an arbitrary open domain in $\mathbb{C}$ and suppose $E$ and $F$ are two disjoint Arakeljan sets in $G$. Then the set $E\cup F$ is an Arakeljan set in $G$.
\end{lem}

\begin{proof}
The sets $E$ and $F$ are both without $G-$holes. Hence, their union is without $G-$holes. Assume for a contradiction that there exists a connected compact subset $K\subset G$ such that $\partial K$ is the union of jordan curves and such that $H$, the set of all $G-$ holes of $E\cup F\cup K$ is either unbounded or $\partial H\cap \partial G\not=\emptyset$. Then by {Remark 2} it follows the set $E\cup F\cup K$ has infinite $G-$holes. Consequently, the set $H$ must have infinitely many components i.e $H=\bigcup_{i=1}^\infty \lbrace h_i\rbrace$ where each $h_i$ represents a $G-$hole of $E\cup F\cup K$. Now for fixed $i$, let $a_k^i,\ \ k=1,2,...$ be all the $G-$holes of $\overline{h_i}\cup K$. The connected compact set $\overline{h_i}\cup\bigcup_k \overline{ a_k^i} $ has no $G-$holes. Denote by $d_i$ the interior of this connected compact set which contains $h_i$. Now $\partial d_i$ consists of an open arc on $\partial K$\footnote{Since $\partial K$ is the union of jordan curves and $K$ is connected} and a connected compact set which we denote by $K_i$. That is, $K_i=\overline{(\partial d_i)\setminus\partial K}.$ Clearly, $K_i\subset E\cup F$. Hence $K_i$ lies completely either on $E$ or in $F$. Let $i_n$,\ \ $n=1,2,...,$ be all the natural numbers for which $K_{i_n}\subset E$ and let $i_l,\ \ l=1,2,...,$ be the remaining natural numbers such that $K_{i_l}\subset F$. Now by our assumption the set $H=\bigcup_{i=1}^\infty \lbrace h_i\rbrace$ is either unbounded or $\partial H\cap \partial G\not=\emptyset$ or both. The case where $H$ is unbounded is already dealt with on \cite{RU}. For the case $\partial H\cap \partial G\not=\emptyset$ we simply note that since $H$ has an accumulation point on $\partial G$, it follows that one or both of the sets $\lbrace h_{i_n}\rbrace $ or $\lbrace h_{i_l}\rbrace$ has an accumulation point on $\partial G$. Without loss of generality let's assume that $\lbrace h_{i_n}\rbrace $ has an accumulation point on $\partial G$. Now let us consider the holes of $\lbrace h_{i_n}\rbrace \cup K$. Since $K_{i_n}\subset E$ and $h_{i_n}$ is a $G-$hole of $E\cup F\cup K$, we see that there exists a $G-$hole $V_{i_n}$ of $E\cup K$ so that $h_{i_n}\subset V_{i_n}\subset d_{i_n}$. Hence the set $\bigcup_{n=1}^\infty\lbrace V_{i_n}\rbrace$ has an accumulation point on $\partial G$, which implies that the set $E$ is not an Arakeljan set. This is a contradiction hence the lemma is proved.
\end{proof}

\indent \indent Now that we have the two lemmas in our arsenal we can prove our man result but first we need to introduce an important definition.

\noindent\textbf{Definition 4.} Let \(G\) be an arbitrary domain in \(\mathbb{C}\) and let \(E\) and \(F\) be relatively closed subsets of \(G\). We say that \(E\) and \(F\) are hole independent if for any \(G-\)unbounded component \(U\) of \(E\) and \(G-\)unbounded component \(V\) of \(F\), the intersection \(U\cap V\) is \(G-\)bounded.

Hole independence is needed in the proof of sufficiency of the man theorem.

\begin{thm}[Main Theorem]
Let $G\subset \mathbb{C}$ be an arbitrary domain and suppose $F$ is a relatively closed subset of $G$, and $C$ is a closed subset in $G$ such that \(F\) and \(C\) are \(G-\)hole independent. In order that every function $f\in A(F)$ can be approximated uniformly on $F$ by functions $g\in Hol(G)$, each of which is bounded on $C$, it is necessary and sufficient that: \begin{enumerate}
\item $F$ is an Arakeljan set.
\item There exists a closed Arakeljan set $C_1\not=G$, so that $C\subset C_1$ and $F\cap C_1$ is a bounded set.
\end{enumerate}
\end{thm}

\begin{proof}[Main Theorem]

\textbf{(Necessity.)} Assume that any $f\in A(F)$ can be approximated uniformly by functions $g\in Hol(G)$, each of which is bounded on $C$. By Arakeljan's Theorem it follows $F$ is an Arakeljan set. Clearly the function $\phi(z)=z$ belongs to the class $Hol(G)$, thus there exists a function $f\in Hol(G)$ such that: \begin{equation}|z-f(z)|<1,\   \ z\in F  
\end{equation}and $f$ is also bounded on C.
 By {lemma 1} there exists a closed Arakeljan set $C_1\subset \mathbb{C}$ such that that $C\subset C_1$ and $f(z)$ is bounded on $C_1$. By $(1)$ clearly $F\cap C_1$ is bounded. The proof of necessity is complete.
 
 \noindent \textbf{(Sufficiency.)} We have Arakeljan sets $F$ and $C_1$, $C\subset C_1$, such that $F\cap C_1$ is a bounded set. Let $K$ be a compact set in $G$ such that $\partial K$ is the union of jordan curves and so large that:
 \begin{equation} 
  F\cap C_1\subset K
   \end{equation}
  \begin{equation}
  (\partial K)\setminus C_1\not=\emptyset
   \end{equation}

\noindent Note that $(3)$ is possible since $C_1\not=G$ is a closed set without $G-$holes.\\
 The set $C_1\setminus K$ is also without $G-$holes, because otherwise there would exist a $G-$hole $g$ of $C_1\setminus K$. This leads to a contradiction since we must have $\partial g\subset C_1$. However, since $C_1$ has no $G-$holes it follows that $g=K$ and this contradicts $(3)$.\\
 The set $C_1\setminus K$ is also an Arakeljan set. Otherwise, there would exist a compact set $L\subset G$ so that the set $H\equiv \lbrace h: h \text{ is a }G-\text{hole of }L\cup(C_1\setminus K)\rbrace$ is either unbounded or $\partial H\cap \partial G\not=\emptyset$. But this is a contradiction because then the set of $G-$holes of the set $C_1\setminus K$ union the compact set $L\cup K$ would also have to either be unbounded or accumulate to a point in $\partial G$. Indeed, the claim follows, since $(L\cup K)\cup (C_1\setminus K)=C_1\cup K$. This contradicts that $C_1$ is an Arakeljan set, as such, by the contradiction we conclude $C_1\setminus K$ is an Arakeljan set.\\
 Now the Arakeljan sets $F$ and $C_1\setminus K$ are disjoint by $(2)$, therefore by the previous lemma we have that $F\cup (C_1\setminus K)$ is an Arakeljan set. \\
 Let $\phi(z)\in A(F)$ be any function. Define a function $h(z)$ by $h(z)=\phi(z)$ on $F$ and $h(z)=0$ on $C_1\setminus K$. Clearly by $(2)$ and the lemmas above it is evident that $f\in A(F\cup (C_1\setminus K))$. Therefore by Arakeljan's theorem $f$ is uniformly approximable by functions in $A(G)$. Therefore for any $\epsilon>0$, there exists a function $f\in A(G)$ such that $|h(z)-f(z)|<\epsilon$ for any $z\in F\cup (C_1\setminus K)$. Hence we have that $|\phi(z)-f(z)|<\epsilon$ on $F$ and $|f(z)|<\epsilon$ on $C_1\setminus K$. The function $f(z)$ is bounded on $C_1\setminus K$, moreover $f(z)$ is bounded on the compact set $C_1\cap K$ thus, $f(z)$ uniformly approximates $\phi(z)$ on $F$ and is bounded on $C$. The proof is complete.
\end{proof}

\pagebreak

\end{document}